\documentclass[12pt]{amsart}
\usepackage{amsmath,amssymb}

\allowdisplaybreaks[1]

\renewcommand{\d}{\displaystyle}

\theoremstyle{plain}

\theoremstyle{plain}
\newtheorem{theorem}{Theorem}
\newtheorem{lemma}{Lemma}

\theoremstyle{remark}
\newtheorem{remark}{Remark}

\numberwithin{theorem}{section}
\numberwithin{lemma}{section}
\numberwithin{corollary}{section}
\numberwithin{remark}{section}
\numberwithin{equation}{section}

\renewcommand{\d}{\displaystyle}

\usepackage{xcolor}
\def\l({\left(}
\def\r){\right)}

\begin{document}

\title{\bf Sums of averages of gcd-sum functions II }
\author{Lisa Kaltenb\"ock, Isao Kiuchi, Sumaia Saad Eddin, Masaaki Ueda}

\maketitle
{\def\thefootnote{}
\footnote{{\it Mathematics Subject Classification 2010: 11A25, 11N37.}}

\begin{abstract}
 Let
$
\gcd(k,j)
$
denote  the greatest common divisor of the integers $k$ and $j$,
and let $r$ be any fixed positive integer. 
Define
$$
M_r(x; f) := \sum_{k\leq x}\frac{1}{k^{r+1}}\sum_{j=1}^{k}j^{r}f(\gcd(j,k))
$$
for  any  large real number $x\geq  5$, where  $f$ is  any arithmetical function.
 Let  $\phi$, and $\psi$ denote the Euler totient and the Dedekind   function, respectively.   
In this paper, we refine asymptotic  expansions of $M_r(x; {\rm id})$, $M_r(x;{\phi})$ and $M_r(x;{\psi})$. 
Furthermore, under the Riemann Hypothesis  and  the simplicity of zeros of the Riemann zeta-function, we establish  the asymptotic formula  of  
$M_r(x;{\rm id})$ for any large positive number $x>5$ satisfying $x=[x]+\frac12$. 
\end{abstract}

\maketitle

 \section{Introduction and Statement of Results}

Let $\gcd(k,j)$ be the greatest common divisor  of the integers $k$ and $j$. The gcd-sum function, which is also known as Pillai's arithmetical function, is defined by 
$$ P(n)=\sum_{k=1}^{n}\gcd (k,n).
$$
This function has been studied by many authors such as  Broughan \cite{Br},  Bordell\'{e}s \cite{B},Tanigawa and  Zhai \cite{TZ},  T\'{o}th \cite{To4}, and others. 
Analytic properties  for partial sums of the gcd-sum function $f(\gcd(j,k))$ were recently studied  by 
Inoue and Kiuchi \cite{IK, IK1}. We recall that the symbol $*$ denotes the Dirichlet convolution of two arithmetical functions $f$ and $g$ defined by $f * g(n) = \sum_{d \mid n} f(d) g(n / d)$, for every positive integer $n$.
For any arithmetical function $f$,  the second author~\cite{K1} showed, that for any fixed positive integer $r$ and any large positive number $x \geq 2$ we have 
\begin{align}                                                              \label{K-formula}
  M_r(x;f) & := \sum_{k\leq x}  \frac{1}{k^{r+1}}\sum_{j=1}^{k}j^{r}f(\gcd(k,j)) \nonumber \\
& =\frac12 \sum_{n\leq x}\frac{f(n)}{n}
  + \frac{1}{r+1}\sum_{d\ell\leq x}\frac{\mu*f(d)}{d}    
 + \frac{1}{r+1} \sum_{m=1}^{[r/2]}{r+1 \choose 2m} B_{2m} \sum_{d\ell\leq x}\frac{\mu*f(d)}{d} \frac{1}{\ell^{2m}}.    
\end{align}
Here, as usual, the function $\mu$ denotes
the M\"{o}bius function and $B_m = B_m(0)$ are the Bernoulli numbers, with $B_{m}(x)$ being the Bernoulli polynomials defined by the generating function
$$
\d \frac{ze^{xz}}{e^{z}-1}=\sum_{m=0}^{\infty}B_{m}(x)\frac{z^m}{m!}
$$
with $|z|<  2\pi$. Many applications of Eq.~\eqref{K-formula} have been given in~\cite{K}, \cite{K2} and \cite{KS}.\\

In \cite{K1}, Eq.~\eqref{K-formula} was used to establish asymptotic formulas for $M_r(x;f)$ for specific choices of $f$ such as the identity function ${\rm id}$, the Euler totient function $\phi={\rm id}*\mu$ or the Dedekind function $\psi ={\rm id}*|\mu|$. More precisely, let $\zeta(s)$ denote the Riemann zeta-function, then for $f = {\rm id}$ it was proved that
\begin{multline*}                                                           \label{gcd-w}
M_{r}(x;{\rm id})   
= \frac{1}{(r+1)\zeta(2)} x\log x + \frac{x}{2}   \\
+ \frac{1}{(r+1)\zeta(2)}\left(2\gamma -1 - \frac{\zeta'(2)}{\zeta(2)} + \sum_{m=1}^{[r/2]}\binom{r+1}{2m} B_{2m} \zeta(2m+1) \right)x  
+ K_{r}(x), 
\end{multline*}
where   
\begin{equation}
\label{K-w}
K_{r}(x) 
= \frac{1}{r+1} \sum_{n\leq x}\frac{\mu(n)}{n}\Delta\left(\frac{x}{n}\right) + O_{r}\left(\log x\right).
\end{equation}
For $f = \phi$ it was shown that
\begin{multline*}                                                        
\label{K21}
M_{r}(x;\phi)    
=\frac{1}{(r+1)\zeta^{2}(2)}x \log x + \frac{x}{2\zeta(2)}    \\
+ \frac{1}{(r+1)\zeta^{2}(2)}\left(2\gamma -1 -2\frac{\zeta'(2)}{\zeta(2)} + \sum_{m=1}^{[r/2]}\binom{r+1}{2m} B_{2m} \zeta(2m+1)\right) x + L_{r}(x),  
\end{multline*}
where 
\begin{equation}                                                           \label{Lr}  
L_{r}(x) 
:= \frac{1}{r+1}\sum_{n\leq x}\frac{\mu*\mu(n)}{n}\Delta\left(\frac{x}{n}\right) + O_{r}\left((\log x)^2\right). 
\end{equation}
Lastly, for $f = \psi$ it was proved that
\begin{multline*}                                                           \label{K31}
M_{r}(x;\psi)  =  \frac{1}{(r+1)\zeta^{}(4)} x\log x  + \frac{\zeta(2)}{2\zeta(4)} x     \\
+ 
\frac{1}{(r+1)\zeta(4)} 
\left(2\gamma - 1 - 2 \frac{\zeta'(4)}{\zeta(4)}+ \sum_{m=1}^{[r/2]}\binom{r+1}{2m}B_{2m} \zeta(2m+1) \right)x  + U_{r}(x),  
\end{multline*}
where  
\begin{equation}                                                           \label{Mr}
U_{r}(x) :=  \frac{1}{r+1} \sum_{n\leq x} \frac{\mu*|\mu|(n)}{n}\Delta\left(\frac{x}{n}\right)+  O_{r}\left((\log x)^{2}\right).
\end{equation}
The function $\Delta(x)$ denotes the error term of the Dirichlet divisor problem: Let $\tau = {\bf 1}*{\bf 1}$ be the divisor function, then for any large positive number $x \geq 2$,
\begin{align}                                                                                                        \label{PP}
\sum_{n\leq x}\tau(n)  &= x\log x + (2\gamma -1)x + \Delta(x),
\end{align}
where $\gamma$ is the Euler constant and $\Delta(x)$ can be estimated by $\Delta(x) = O\l(x^{\theta + \varepsilon}\r)$. It is known that  one can take  $1/4\leq \theta \leq 1/3$.

The first purpose of this paper is to refine the error terms $K_r(x), L_r(x)$ and $U_r(x)$ from the above formulas. Therefore, let $\sigma_{u} = {\rm id}_{u}*{\bf 1}$ be the generalized divisor function for any real number $u$ and let $m\geq 1$ be an integer. Then for any large positive number $x\geq 2$, 
the function $\Delta_{-2m}(x)$ denotes the error term of the generalized divisor problem given by
\begin{align}    
\label{QQ}
\sum_{n\leq x}\sigma_{-2m}(n) &= \zeta(1+2m)x - \frac{1}{2}\zeta(2m) + \Delta_{-2m}\l(x\r). 
\end{align}
We have the following results:


\begin{theorem}                                    
\label{th11}
Let $\Delta(x)$ and $\Delta_{-2m}(x)$ be the error terms  given by Eqs.~\eqref{PP} and \eqref{QQ}, respectively.
For any large  positive  number $x > 5$ and fixed positive  integer $r$, we have
\begin{align*}      															
K_r(x)
&= \frac{1}{r+1} \sum_{d\leq x} \frac{\mu(d)}{d}\Delta\l(\frac{x}{d}\r) 
 + \frac{1}{r+1}\sum_{d\leq x} \frac{\mu(d)}{d} \sum_{m=1}^{[r/2]}{r+1 \choose 2m} B_{2m} \Delta_{-2m}\l(\frac{x}{d}\r)   
 + O_{r}\l(\delta(x)\log x\r),                                       
\end{align*}
where the function $\delta(x)$ is defined by
\begin{align}                                            \label{delta}
\delta(x) := {\rm{exp}}\l(-C\frac{(\log x)^{3/5}}{(\log \log x)^{1/5}}\r)
\end{align}
with $C$ being a positive constant.

Moreover, we have
	\begin{align*}                                                              
	L_r(x) &=  \frac{1}{r+1} \sum_{n\leq x} \frac{\mu*\mu(n)}{n}\Delta\l(\frac{x}{n}\r)             \\
	& \quad  + \frac{1}{r+1} 
	\sum_{n\leq x} \frac{\mu*\mu(n)}{n}\sum_{m=1}^{[r/2]}{r+1 \choose 2m}B_{2m}\Delta_{-2m}\l(\frac{x}{n}\r)   
	+  O_{r}\l((\log x)^{2/3}(\log\log x)^{1/3}\r).                                          \nonumber
	\end{align*}
	and 
	\begin{align*}                                                               
	&U_r(x) =  \frac{1}{r+1} \sum_{n\leq x} \frac{\mu*|\mu|(n)}{n}\Delta\l(\frac{x}{n}\r)             \\
	& \quad  + \frac{1}{r+1} 
	\sum_{n\leq x} \frac{\mu*|\mu|(n)}{n}\sum_{m=1}^{[r/2]}{r+1 \choose 2m}B_{2m}\Delta_{-2m}\l(\frac{x}{n}\r)   
	- \frac{1}{4\zeta(2)}\log x +  O_{r}\l((\log x)^{2/3}\r).                  \nonumber
	\end{align*}
\end{theorem}

\begin{remark}
It is easily checked that using the weakest estimate $\Delta_{-2m}(x)=O_{m}(1)$ in the results Theorem~\ref{th11} yields the previously known formulas for $K_r(x)$, $L_r(x)$ and $M_r(x)$ from Eqs. \eqref{K-w}, \eqref{Lr}, and \eqref{Mr}.
\end{remark}


Furthermore, even better estimates of $K_r(x)$ can be achieved by additional assumptions on the Riemann zeta-function. Under the Riemann Hypothesis, a sharper estimate of the partial sum of the M\"{o}bius function has been given by Soundararajan \cite{So}, who proved that  
\begin{align*}                                                              
M(x) := \sum_{n\leq x}\mu(n) = O\l({x^{1/2}}{\eta(x)}\r)
\end{align*}
where 
\begin{align}                                                              \label{eta}
\eta(x) := {\rm exp}\l((\log x)^{1/2}(\log\log x)^{14}\r),
\end{align}
for any large positive number $x> 5$ satisfying $x=[x]+\frac12$.
This result has later been improved by Maier and Montgomery~\cite{MM} and by Balazard and de Roton~\cite{BR}.
By using the above result on $M(x)$, we obtain the next statement.


\begin{theorem}                                   
\label{th12}
Assume the Riemann Hypothesis and let $\Delta(x)$ and $\Delta_{-2m}(x)$ be the error terms  given by Eqs.~\eqref{PP} and \eqref{QQ}, respectively. Then for any large positive number $x > 5$ such that $x=[x]+\frac12$ and fixed positive  integer $r$,  we  have  
\begin{align*}                                                              
K_r(x) &= \frac{1}{r+1} \sum_{d\leq x} \frac{\mu(d)}{d}\Delta\l(\frac{x}{d}\r)               \\
&  + \frac{1}{r+1}  \sum_{d\leq x} \frac{\mu(d)}{d}\sum_{m=1}^{[r/2]}{r+1 \choose 2m} B_{2m} \Delta_{-2m}\l(\frac{x}{d}\r) 
   + O_{r}\l(\frac{ \eta(x) \log x}{x^{1/2}} \r).                                     \nonumber 
\end{align*}
\end{theorem}


For our further considerations, let $\rho=\beta +i\gamma$ denote the generic non-trivial zeros of the Riemann zeta-function. Under the assumption that all zeros $\rho$ in the critical strip of $\zeta(s)$ are simple, we are able to prove an additional refinement for the error term $K_r(x)$.


\begin{theorem}                                   
	\label{th13}
	Assume that the zeros of $\zeta(s)$ are simple. Let $T_{*}\geq x^{6}$ be some positive number satisfying the inequality
	\begin{align*}                                                              
	\frac{1}{\zeta(\sigma+iT_{*})} \ll T_{*}^{\varepsilon}
	\end{align*}
	for $\frac12 \leq \sigma \leq 2$. For any large positive number $x> 5$ with $x=[x]+\frac12$ we then have
	\begin{align*}                                                               
	K_r(x)
	& =  \frac{1}{r+1}\sum_{n\leq x}\frac{\mu(n)}{n}\Delta\l(\frac{x}{n}\r) 
	+ \frac{1}{r+1}\sum_{n\leq x}\frac{\mu(n)}{n}\sum_{m=1}^{[r/2]} {r+1 \choose 2m} B_{2m} \Delta_{-2m}\l(\frac{x}{n}\r)  \\  
	&  + \frac{2\gamma + C_{\rm{odd}}(r) -1 }{r+1}\sum_{|\gamma|\leq T_{*}}\frac{x^{\rho-1}}{(\rho-2)\zeta'(\rho)}      
	- \frac{C_{\rm{even}}(r)}{2(r+1)}  \sum_{|\gamma|\leq T_{*}}\frac{x^{\rho-1}}{(\rho-1)\zeta'(\rho)}  \nonumber \\
	& + \frac{1}{r+1}\sum_{|\gamma|\leq T_{*}}\frac{x^{\rho-1}}{(\rho-2)^2 \zeta'(\rho)}                        
	+ O_{r}\l(x^{-3}\r),                                           \nonumber
	\end{align*}
	where the functions $C_{\rm{odd}}(r)$ and $ C_{\rm{even}}(r)$ are given by 
	\begin{align*}                                                              
	C_{\rm{odd}}(r) := \sum_{m=1}^{[r/2]}{r+1 \choose 2m} B_{2m} \zeta(2m+1),
	\end{align*}
	and
	\begin{align*}                                                               
	C_{\rm{even}}(r) := \sum_{m=1}^{[r/2]}{r+1 \choose 2m} B_{2m} \zeta(2m)
	\end{align*}
	for any fixed positive integer $r$.
\end{theorem}


Finally, define the sum 
$$
J_{-\lambda}(T) := \sum_{0 < \gamma \leq T} \frac{1}{|\zeta'(\rho)|^{2\lambda}} 
$$
which is  intimately connected to Mertens function. Assuming the simplicity of the zeros of $\zeta(s)$,  Gonek \cite{G} and  Hejhal \cite{H} independently conjectured  that for any real number $\lambda < 3/2$, we have
\begin{align}
J_{-\lambda}(T)& \asymp T(\log{T})^{(\lambda -1)^2}                      \label{GH}
\end{align} 
We use this conjecture to prove the following:


\begin{theorem}                                                             \label{th33}
	Assume that the Riemann Hypothesis and Gonek-Hejhal conjecture. Then
	\begin{align*}                                                               
	K_r(x)
	&=  \frac{1}{r+1}\sum_{n\leq x}\frac{\mu(n)}{n}\Delta\l(\frac{x}{n}\r)                                           
	+   \frac{1}{r+1}\sum_{d\leq x}\frac{\mu(d)}{d}\sum_{m=1}^{[r/2]}{r+1 \choose 2m} B_{2m}\Delta_{-2m}\l(\frac{x}{d}\r)   \\
	& + O_{r}\l(\frac{(\log x)^{5/4}}{x^{1/2}}\r),    \nonumber 
	\end{align*}
	for any large positive number  $x>  5$ satisfying $x=[x]+\frac12$. 
\end{theorem}

\section{Proofs of Theorems  \ref{th11} and \ref{th12}}

In order to prove our main results, we first show some necessary lemmas.

\subsection{Auxiliary lemmas}


\begin{lemma}          \label{lem21}
For  any large  positive  number $x > 5$,  we have
\begin{align}
 \sum_{n\leq x}\frac{\mu(n)}{n^2}  &= \frac{1}{\zeta(2)} + O\l(\frac{\delta(x)}{x}\r),                                  \label{delta11} \\
 \sum_{n\leq x}\frac{\mu(n)}{n^2} \log n  &= \frac{\zeta'(2)}{\zeta^{2}(2)}+ O\l(\frac{\delta(x)}{x}\r),                \label{delta21}
\end{align}
and
\begin{align}
\sum_{n\leq x}\frac{\mu(n)}{n} &=  O\l(\delta(x)\r),                                                               \label{delta41}
\end{align}
where $\delta(x)$ is given by Eq.~\eqref{delta}. 
Assume that $x=[x]+\frac{1}{2}$. Under the Riemann Hypothesis we have
\begin{align}
\sum_{n\leq x}\frac{\mu(n)}{n^2}  &= \frac{1}{\zeta(2)} + O\l(\frac{\eta(x)}{x^{3/2}}\r),                               \label{eta11}  \\
 \sum_{n\leq x}\frac{\mu(n)}{n^2} \log n &= \frac{\zeta'(2)}{\zeta^{2}(2)}+ O\l(\frac{\eta(x)\log x}{x^{3/2}}\r),       \label{eta21}
\end{align}
and
\begin{align}
 \sum_{n\leq x}\frac{\mu(n)}{n} &= O\l(\frac{\eta(x)}{x^{1/2}}\r).             \label{eta41}
\end{align}
for any large  positive  number $x > 5$. Here $\eta(x)$ is given by Eq.~\eqref{eta}.  
\end{lemma}

\begin{proof}
Eqs.~\eqref{delta11} and \eqref{delta21} follow from  Lemmas 2.2 and 2.3 in \cite{SS}. The proof of Eq.~\eqref{delta41} can be found in \cite{Jia}. 
The formulas \eqref{eta11}--\eqref{eta41} follow from  Lemma 3.1 in~\cite{IK}.                                    
\end{proof}


\begin{lemma}          
\label{lem20}
For any large  positive  number $x > 5$, we have 
\begin{align*}
 \sum_{n\leq x}\frac{\phi(n)}{n}  &= \frac{x}{\zeta(2)} +  O\l((\log x)^{2/3}(\log\log x)^{1/3}\r).       
\end{align*}
\end{lemma}

\begin{proof}  
For any large positive  number $x\geq 5$, we use  the result  of Liu  in \cite{L}
$$
\sum_{\ell\leq x}\frac{\mu(\ell)}{\ell}\vartheta\l(\frac{x}{\ell}\r) = O\l((\log x)^{2/3}(\log\log x)^{1/3}\r),
$$
the fact that $ \phi={\rm id}*\mu$, and  Eqs.~\eqref{delta11}, \eqref{delta41} to obtain the formula
\begin{align*}
\sum_{n\leq x}\frac{\phi(n)}{n} &= \sum_{\ell\leq x}\frac{\mu(\ell)}{\ell}\l(\frac{x}{\ell}- \vartheta\l(\frac{x}{\ell}\r) -\frac12\r)
\\
 &= \frac{x}{\zeta(2)} + O\l((\log x)^{2/3}(\log\log x)^{1/3}\r),
\end{align*}
where $\vartheta(x)$ is the oscillatory function defined by $x-[x]-\frac12$. This completes the proof. 
\end{proof}


\begin{lemma}   \label{lem201}
For any large  positive  number $x > 5$, we have 
\begin{align*}
\sum_{n\leq x}\frac{\psi(n)}{n}&=\frac{\zeta(2)}{\zeta(4)}x  - \frac{1}{2\zeta(2)}\log x
  + O\l((\log x)^{2/3}\r).                                                                    
\end{align*}
\end{lemma}

\begin{proof}  
The proof can be found in~\cite[Satz 3]{W}.  
\end{proof}


\begin{lemma}          
\label{lem211}
For  any large  positive  number $x > 5$,  we have
\begin{align}
 \sum_{n\leq x}\frac{\mu*\mu(n)}{n^2}  &= \frac{1}{\zeta^{2}(2)} + O\l(\frac{\delta(x)}{x}\r),                             
 \label{delta111} \\
 \sum_{n\leq x}\frac{\mu*\mu(n)}{n^2} \log n  &= 2\frac{\zeta'(2)}{\zeta^{3}(2)}+ O\l(\frac{\delta(x)}{x}\r),               \label{delta211}
\end{align}
and
\begin{align}
\sum_{n\leq x}\frac{\mu*\mu(n)}{n} &=  O_{}\l(\delta(x)\r).                    \label{delta411}
\end{align}
\end{lemma}

\begin{proof}  
Eqs.~\eqref{delta111} and \eqref{delta211} follow from Eqs.~(1.13) and (1.14) in~\cite{IK1}, respectively. 
For Eq.~\eqref{delta411}, we use Eq.~\eqref{delta41} to get 
\begin{align*}
&   \sum_{n\leq x}\frac{\mu*\mu(n)}{n}
 = \sum_{d\leq x}\frac{\mu(d)}{d}\sum_{\ell\leq x/d}\frac{\mu(\ell)}{\ell}  
 =  O\l(\delta(x)\r).
\end{align*}
This completes the proof. 
\end{proof}


\begin{lemma}          
\label{lem211s}
For  any large   positive  number $x > 5$,  we have
\begin{align}
 \sum_{n\leq x}\frac{|\mu|*\mu(n)}{n^2}  &= \frac{1}{\zeta^{}(4)} + O\l(\frac{\delta(x)}{x}\r),                              \label{delta111w} \\
 \sum_{n\leq x}\frac{|\mu|*\mu(n)}{n^2} \log n  &= 2\frac{\zeta'(4)}{\zeta^{2}(4)}+ O\l(\frac{\delta(x)}{x}\r),               \label{delta211w}
\end{align}
and
\begin{align}
\sum_{n\leq x}\frac{|\mu|*\mu(n)}{n} &=  \frac{1}{\zeta(2)} +O_{}\l(\delta(x)\r).                                             \label{delta411w}
\end{align}
\end{lemma}

\begin{proof}  
Eqs.~\eqref{delta111w} and \eqref{delta211w} follow from Eqs.~(1.17) and (1.18) in~\cite{IK1}, respectively. 
It is known that 
$$
 \sum_{n=1}^{\infty}\frac{|\mu|*\mu(n)}{n}=  \frac{1}{\zeta(2)},
$$
Now, we write our sums as follows
\begin{align*}
\sum_{n\leq x}\frac{|\mu|*\mu(n)}{n}&=\sum_{n=1}^{\infty}\frac{|\mu|*\mu(n)}{n}-\sum_{n> x}\frac{|\mu|*\mu(n)}{n}\\ &=
\frac{1}{\zeta(2)}-\sum_{n >x}\frac{|\mu|*\mu(n)}{n}.
\end{align*}
To complete the proof, it remains to estimate the last sum above. Notice that 
\begin{align*}
\sum_{n> x}\frac{|\mu|*\mu(n)}{n}=\int_ {x}^{\infty}\frac{\sum_{x<n\leq t}|\mu|*\mu (n)}{t^2}\, dt
\end{align*}
and that
\begin{align*}
\sum_{n\leq x}|\mu|*\mu(n) = \sum_{d\leq x}|\mu(d)|\sum_{\ell\leq x/d}\mu(\ell) 
 = O\l(x\sum_{d\leq x}\frac{|\mu(d)|}{d}\delta\l(\frac{x}{d}\r)\r) = O\l(x\delta(x)\r).
\end{align*}
Therefore, we have
\begin{align*}
\sum_{n> x}\frac{|\mu|*\mu(n)}{n}&= O\l(\int_{x}^{\infty} \frac{t \delta(t)}{t^2}\, dt \r)+O\l(\delta(x)\r)\\ &=
O\l(\delta(x)\r),
\end{align*}
and Eq.~\eqref{delta411w} is proved.
\end{proof}


\begin{lemma}          
\label{lem22}
For  any large  positive  number $x> 5$,  we have
\begin{equation}														\label{phi12}
\begin{split}
\sum_{n\leq x}\sum_{d|n}\frac{\phi(d)}{d}  &= \frac{1}{\zeta(2)}x\log x
+ \frac{1}{\zeta(2)}\l(2\gamma -1 - \frac{\zeta'(2)}{\zeta(2)}\r)x  \\
& + \sum_{d\leq x}\frac{\mu(d)}{d}\Delta\l(\frac{x}{d}\r)  + O\l(\delta(x)\log x\r),   
\end{split}
\end{equation}
and
\begin{align}                                                               \label{phi22}
\sum_{d\ell\leq x}\frac{\phi(d)}{d}\frac{1}{\ell^{2m}}
&= \frac{\zeta(1+2m)}{\zeta(2)}x + \sum_{d\leq x}\frac{\mu(d)}{d}\Delta_{-2m}\l(\frac{x}{d}\r) 
 + O_{m}\l(\delta(x)\r)                                    
\end{align}
for any positive integer  $m$.
Suppose that $x=[x]+\frac12$. Under the Riemann Hypothesis, we have
\begin{equation}															\label{phi32}
\begin{split}
 \sum_{n\leq x}\sum_{d|n}\frac{\phi(d)}{d}   &= \frac{1}{\zeta(2)}x\log x
+ \frac{1}{\zeta(2)}\l(2\gamma -1 - \frac{\zeta'(2)}{\zeta(2)}\r)x  \\
& + \sum_{d\leq x}\frac{\mu(d)}{d}\Delta\l(\frac{x}{d}\r)  + O\l(\frac{\eta(x)\log x}{x^{1/2}}\r),
\end{split}
\end{equation}
 and
\begin{align}
  \sum_{d\ell\leq x}\frac{\phi(d)}{d}\frac{1}{\ell^{2m}}
 &= \frac{\zeta(1+2m)}{\zeta(2)}x + \sum_{d\leq x}\frac{\mu(d)}{d}\Delta_{-2m}\l(\frac{x}{d}\r)
+ O_{m}\l(\frac{\eta(x)}{x^{1/2}}\r)                                           \label{phi42}
\end{align}
\end{lemma}

\begin{proof}
We recall the identity 
$
\d \frac{\phi}{\rm id}*{\bf 1}=\frac{\mu}{\rm id}*\tau.
$
Using Eqs.~\eqref{PP},  \eqref{delta11} and \eqref{delta21}, we obtain
\begin{align*}
&\sum_{n\leq x}\sum_{d|n}\frac{\phi(d)}{d}
= \sum_{d\leq x}\frac{\mu(d)}{d} \sum_{\ell\leq x/d}\tau(\ell) \nonumber  \\
&= x \l(\log x + 2\gamma -1 \r)  \sum_{d\leq x} \frac{\mu(d)}{d^2}  - x \sum_{d\leq x}\frac{\mu(d)}{d^2} \log d
+ \sum_{d\leq x}\frac{\mu(d)}{d}\Delta\l(\frac{x}{d}\r)                                                          \\ 
&=\frac{1}{\zeta(2)}x\log x
+ \frac{1}{\zeta(2)}\l(2\gamma -1 - \frac{\zeta'(2)}{\zeta(2)}\r)x
+ \sum_{d\leq x}\frac{\mu(d)}{d}\Delta\l(\frac{x}{d}\r)  + O\l(\delta(x)\log x\r), \nonumber
\end{align*}
which completes the proof of Eq.~\eqref{phi12}.
Further, we recall the identity  
$
\d \frac{\phi}{\rm id}*{\rm id}_{-2m} = \frac{\mu}{\rm id}*\sigma_{-2m},
$
and use Eqs.~\eqref{QQ}, \eqref{delta11} and \eqref{delta41} to get
\begin{align*}                                                           
\sum_{d\ell\leq x} \frac{\phi(d)}{d}\frac{1}{\ell^{2m}}
&= \sum_{d\leq x}\frac{\mu(d)}{d} \sum_{\ell\leq x/d}\sigma_{-2m}(\ell) \\
&=
\sum_{d\leq x}\frac{\mu(d)}{d}\left( \zeta(1+2m)\frac{x}{d}-\frac{1}{2}\zeta(2m)+\Delta_{-2m}\l(\frac{x}{d}\r)\right)
\\
&=\frac{\zeta(1+2m)}{\zeta(2)}x + \sum_{d\leq x}\frac{\mu(d)}{d}\Delta_{-2m}\l(\frac{x}{d}\r)
 + O_{m}\l(\delta(x)\r). \nonumber
\end{align*}
This completes the proof of Eq.~\eqref{phi22}.  
Similarly, we use Eqs.~\eqref{eta11}, \eqref{eta21} and \eqref{eta41}  to deduce Eqs.~\eqref{phi32}  and \eqref{phi42}.
\end{proof}


\begin{lemma}          \label{lem23}
For  any large  positive  number $x > 5$, we have
\begin{equation}							 \label{phi23}
\begin{split}
\sum_{n\leq x}\sum_{d|n}\frac{\mu*\phi(d)}{d}  
&= \frac{1}{\zeta^{2}(2)}x\log x + \frac{1}{\zeta^{2}(2)}\l(2\gamma -1 - 2\frac{\zeta'(2)}{\zeta(2)}\r)x  \\
& + \sum_{d\leq x}\frac{\mu*\mu(d)}{d}\Delta\l(\frac{x}{d}\r)  + O\l(\delta(x)\log x \r), 
\end{split}                                   
\end{equation}
and
\begin{align}                                                                  \label{phi33}
\sum_{d\ell\leq x}\frac{\mu*\phi(d)}{d}\frac{1}{\ell^{2m}}
&= \frac{\zeta(1+2m)}{\zeta^{2}(2)}x + \sum_{d\leq x}\frac{\mu*\mu(d)}{d}\Delta_{-2m}\l(\frac{x}{d}\r)  
+ O_{m}\l(\delta(x)\r)                                                          
\end{align}
for any positive integer $m$.
\end{lemma}

\begin{proof}
We  use the identity  
$
\d \frac{\mu*\phi}{\rm id}*{\bf 1}=\frac{\mu*\mu}{\rm id}*\tau,
$
Eqs.~\eqref{PP}, \eqref{delta111}, and \eqref{delta211} to obtain
\begin{align*}
&\sum_{k\leq x}\sum_{d|k}\frac{\mu*\phi(d)}{d}
= \sum_{d\leq x}\frac{\mu*\mu(d)}{d} \sum_{\ell\leq x/d}\tau(\ell)   \\
&= x \l(\log x + 2\gamma -1 \r)  \sum_{d\leq x} \frac{\mu*\mu(d)}{d^2}  
   - x\sum_{d\leq x}\frac{\mu*\mu(d)}{d^2} \log d
+ \sum_{d\leq x}\frac{\mu*\mu(d)}{d}\Delta\l(\frac{x}{d}\r)           \\
&=\frac{1}{\zeta^{2}(2)}x\log x
+ \frac{1}{\zeta^{2}(2)}\l(2\gamma -1 - 2\frac{\zeta'(2)}{\zeta(2)}\r)x  
 + \sum_{d\leq x}\frac{\mu*\mu(d)}{d}\Delta\l(\frac{x}{d}\r)  + O\l(\delta(x)\log x\r),  
\end{align*}
which completes  the proof of Eq.~\eqref{phi23}.
By using the fact that
$
\d \frac{\mu*\phi}{\rm id}*{\rm id}_{-2m} = \frac{\mu*\mu}{\rm id}*\sigma_{-2m},
$
together with Eqs.~\eqref{QQ}, \eqref{delta111}, and \eqref{delta411} we get
\begin{align*}                                   
\sum_{d\ell\leq x} \frac{\mu*\phi(d)}{d}\frac{1}{\ell^{2m}}
&= \sum_{d\leq x}\frac{\mu*\mu(d)}{d} \sum_{\ell\leq x/d}\sigma_{-2m}(\ell)   \\
&=\frac{\zeta(1+2m)}{\zeta^{2}(2)}x + \sum_{d\leq x}\frac{\mu*\mu(d)}{d}\Delta_{-2m}\l(\frac{x}{d}\r)
 + O_{m}\l(\delta(x)\r). \nonumber
\end{align*}
Therefore, Eq.~\eqref{phi33} is proved.
\end{proof}


\begin{lemma}  															   \label{lem23q}
For  any large  positive  number $x>  5$,  we have
\begin{equation}										\label{psi23}
\begin{split}
\sum_{n\leq x}\sum_{d|n}\frac{\mu*\psi(d)}{d} 
&= \frac{1}{\zeta^{}(4)}x\log x + \frac{1}{\zeta^{}(4)}\l(2\gamma -1 - 2\frac{\zeta'(4)}{\zeta(4)}\r)x  \\
&  + \sum_{d\leq x}\frac{|\mu|*\mu(d)}{d}\Delta\l(\frac{x}{d}\r)  + O\l(\delta(x)\log x\r),      
\end{split}                               
\end{equation}
and
\begin{align}
\sum_{d\ell\leq x}\frac{\mu*\psi(d)}{d}\frac{1}{\ell^{2m}}
&= \frac{\zeta(1+2m)}{\zeta^{}(4)}x + \sum_{d\leq x}\frac{|\mu|*\mu(d)}{d}\Delta_{-2m}\l(\frac{x}{d}\r)  
 - \frac{\zeta(2m)}{2\zeta(2)} + O_{m}\l(\delta(x)\r)                           \label{psi33}
\end{align}
for any positive integer $m$.
\end{lemma}

\begin{proof}
From  the identity  
$
\d \frac{\mu*\psi}{\rm id}*{\bf 1}=\frac{\mu*|\mu|}{\rm id}*\tau,
$ 
we have 
$$
\sum_{k\leq x}\sum_{d|k}\frac{\mu*\psi(d)}{d}
= \sum_{d\leq x}\frac{\mu*|\mu|(d)}{d} \sum_{\ell\leq x/d}\tau(\ell).  
$$
Using Eqs.~\eqref{PP}, \eqref{delta111w}  and \eqref{delta211w}, we obtain
the formula Eq.~\eqref{psi23}.
Now, we use the identity 
 $
\d \frac{\mu*\psi}{\rm id}*{\rm id}_{-2m} = \frac{\mu*|\mu|}{\rm id}*\sigma_{-2m}
$  
to write our second sums as follows 
$$
\sum_{d\ell\leq x} \frac{\mu*\psi(d)}{d}\frac{1}{\ell^{2m}}
= \sum_{d\leq x}\frac{\mu*|\mu|(d)}{d} \sum_{\ell\leq x/d}\sigma_{-2m}(\ell).  
$$
Again, we use Eq.~\eqref{QQ} to get 
$$
\sum_{d\ell\leq x} \frac{\mu*\psi(d)}{d}\frac{1}{\ell^{2m}}
= 
\sum_{d\leq x}\frac{\mu*|\mu|(d)}{d}\left( \zeta(1+2m)\frac{x}{d}-\frac{1}{2}\zeta(2m)+\Delta_{-2m}\l(\frac{x}{d}\r)\right).
$$
Applying Eqs.~\eqref{delta111w} and \eqref{delta411w} to the above, we deduce
the desired result.
\end{proof}


Now we are ready to prove our main theorems. 
\subsection{Proofs of the Theorems}

\begin{proof}[Proof of Theorem \ref{th11}]
First, we take $f={\rm id}$ into Eq.~\eqref{K-formula} to get
\begin{align}
M_r(x; {\rm id})
& =
\frac{1}{2}\sum_{n\leq x}1 + 
\frac{1}{r+1} \sum_{d\ell\leq x}\frac{\mu*{\rm id} (d)}{d}+
\frac{1}{r+1}\sum_{m=1}^{[r/2]}{r+1 \choose 2m} B_{2m} \sum_{d\ell\leq x}\frac{\mu*{\rm id}(d)}{d}\frac{1}{\ell^{2m}} \nonumber
\\ &=
\frac{1}{2}\sum_{n\leq x}1  + \frac{1}{r+1} \sum_{n\leq x}\sum_{d|n}\frac{\phi(d)}{d}   
+ \frac{1}{r+1}\sum_{m=1}^{[r/2]}{r+1 \choose 2m} B_{2m} \sum_{d\ell\leq x}\frac{\phi(d)}{d}\frac{1}{\ell^{2m}}.   \label{K}
\end{align}
Applying Eqs.~\eqref{phi12} and \eqref{phi22} above yields
\begin{align*}                                                              
K_r(x)
&= \frac{1}{r+1} \sum_{d\leq x} \frac{\mu(d)}{d}\Delta\l(\frac{x}{d}\r) 
 \\&+ \frac{1}{r+1}\sum_{d\leq x} \frac{\mu(d)}{d} \sum_{m=1}^{[r/2]}{r+1 \choose 2m} B_{2m} \Delta_{-2m}\l(\frac{x}{d}\r)   
 + O_{r}\l(\delta(x)\log x\r),                  \end{align*}
which gives the desired result. 
We take $f=\phi$ into Eq.~\eqref{K-formula} to get
\begin{align*}
M_r(x; \phi)
& = \frac{1}{2}\sum_{n\leq x}\frac{\phi(n)}{n}  
+ \frac{1}{r+1} \sum_{n\leq x}\sum_{d|n}\frac{\mu*\phi(d)}{d}   \nonumber \\
& + \frac{1}{r+1} \sum_{m=1}^{[r/2]}{r+1 \choose 2m} B_{2m} \sum_{d\ell\leq x}\frac{\mu*\phi(d)}{d}\frac{1}{\ell^{2m}}.
\end{align*}
Using Lemma  \ref{lem20}, as well as Eqs. \eqref{phi23} and \eqref{phi33}, we get 
\begin{align}                                                              
L_r(x) &=  \frac{1}{r+1} \sum_{n\leq x} \frac{\mu*\mu(n)}{n}\Delta\l(\frac{x}{n}\r) \nonumber            \\
& \quad  + \frac{1}{r+1} 
\sum_{n\leq x} \frac{\mu*\mu(n)}{n}\sum_{m=1}^{[r/2]}{r+1 \choose 2m}B_{2m}\Delta_{-2m}\l(\frac{x}{n}\r)   
+  O_{r}\l((\log x)^{2/3}(\log\log x)^{1/3}\r),                                          \nonumber
\end{align}
as desired. Taking $f=\psi$ into Eq.~\eqref{K-formula} we get 
\begin{align}
M_r(x; \psi)
& = \frac{1}{2}\sum_{n\leq x}\frac{\psi(n)}{n}  
+ \frac{1}{r+1} \sum_{d\ell\leq x}\frac{\mu*\psi(d)}{d}   \nonumber \\
& + \frac{1}{r+1} \sum_{m=1}^{[r/2]}{r+1 \choose 2m} B_{2m} \sum_{d\ell\leq x}\frac{\mu*\psi(d)}{d}\frac{1}{\ell^{2m}}. \nonumber     
\end{align}
Applying Lemma \ref{lem201}, as well as Eqs. \eqref{psi23} and \eqref{psi33} in the above formula yields 
\begin{align}                                                               
&U_r(x) =  \frac{1}{r+1} \sum_{n\leq x} \frac{\mu*|\mu|(n)}{n}\Delta\l(\frac{x}{n}\r)  \nonumber           \\
& \quad  + \frac{1}{r+1} 
\sum_{n\leq x} \frac{\mu*|\mu|(n)}{n}\sum_{m=1}^{[r/2]}{r+1 \choose 2m}B_{2m}\Delta_{-2m}\l(\frac{x}{n}\r)   
- \frac{1}{4\zeta(2)}\log x +  O_{r}\l((\log x)^{2/3}\r).                  \nonumber
\end{align}
This completes the proof of Theorem~\ref{th11}.
\end{proof}


\begin{proof}[Proof of Theorem \ref{th12}]
By assuming the Riemann Hypothesis, and applying Eqs.~\eqref{phi32} and \eqref{phi42} in Eq.~\eqref{K}, we immediately deduce that
\begin{align*}                                                              
K_r(x) &= \frac{1}{r+1} \sum_{d\leq x} \frac{\mu(d)}{d}\Delta\l(\frac{x}{d}\r)               \\
&  + \frac{1}{r+1}  \sum_{d\leq x} \frac{\mu(d)}{d}\sum_{m=1}^{[r/2]}{r+1 \choose 2m} B_{2m} \Delta_{-2m}\l(\frac{x}{d}\r) 
   + O_{r}\l(\frac{ \eta(x) \log x}{x^{1/2}} \r).                                     \nonumber 
\end{align*}
which completes the proof of Theorem~\ref{th12}.
\end{proof}

\section{Proofs of Theorems \ref{th13} and \ref{th33}}

To prove Theorems \ref{th13} we just need the following lemma.

\begin{lemma}                    
\label{lem33}
Under the hypotheses of Theorem~\ref{th13}, we have
\begin{align*}                                                                 
\sum_{n\leq x}\frac{\mu(n)}{n^2}
 &=\frac{1}{\zeta(2)}
 + \sum_{|\gamma|\leq T_{*}}\frac{x^{\rho-2}}{(\rho-2)\zeta'(\rho)} + \frac{\pi^2}{\zeta(3)}x^{-4} + O\l(x^{-5}\r),  
\end{align*}
\begin{align*}                                                                  
\sum_{n\leq x}\frac{\mu(n)}{n^2}\log \frac{x}{n}
 &=\frac{1}{\zeta(2)}\l(\log x - \frac{\zeta'(2)}{\zeta^{}(2)}\r)
 + \sum_{|\gamma|\leq T_{*}}\frac{x^{\rho-2}}{(\rho-2)^{2}\zeta'(\rho)}  
 - \frac{\pi^2}{4\zeta(3)}x^{-4}  + O\l(x^{-5}\r),   
\end{align*}
and
\begin{align*}                                                                  
\sum_{n\leq x}\frac{\mu(n)}{n}
 &=  \sum_{|\gamma|\leq T_{*}}\frac{x^{\rho-1}}{(\rho-1)\zeta'(\rho)}
    +  \frac{4 \pi^2}{3\zeta(3)}x^{-3} +  O\l(x^{-5}\r). 
\end{align*}
\end{lemma}

\begin{proof}
The proof of the lemma can be found in~\cite{IK}.
\end{proof}


\begin{proof}[Proof of Theorem  \ref{th13}]
We recall that 
\begin{align*}
M_r(x; {\rm id})
    & = \frac{1}{2}\sum_{n\leq x}1  + \frac{1}{r+1} \sum_{d\ell\leq x}\frac{\mu* \rm id(d)}{d}   
+ \frac{1}{r+1}\sum_{m=1}^{[r/2]}{r+1 \choose 2m} B_{2m} \sum_{d\ell\leq x}\frac{\mu*\rm id(d)}{d}\frac{1}{\ell^{2m}}.          
\end{align*}
Using the fact that 
$$ \frac{\mu* \rm id}{\rm id}* {\bf{1}}= \frac{\mu}{\rm id}* \tau, \qquad \qquad  \frac{\mu* \rm id}{\rm id}*{\rm id}_{-2m}=  \frac{\mu}{\rm id}*\sigma_{-2m},$$ 
and Eqs.~(\ref{PP}) and (\ref{QQ}), we get 
\begin{align*}
M_r(x, {\rm id})  & = \frac{[x]}{2}
 +  \frac{x}{r+1} \sum_{n\leq x}\frac{\mu(n)}{n^2}\log \frac{x}{n}       
 + \frac{x}{r+1}\l(2\gamma-1 + C_{\rm{odd}}(r) \r)\sum_{n\leq x}\frac{\mu(n)}{n^2}    \nonumber  \\
&  - \frac{C_{\rm{even}}(r)}{2(r+1)} \sum_{n\leq x}\frac{\mu(n)}{n}
  + \frac{1}{r+1}\sum_{n\leq x}\frac{\mu(n)}{n}\Delta\l(\frac{x}{n}\r)  \nonumber \\
&  + \frac{1}{r+1}\sum_{m=1}^{[r/2]} {r+1 \choose 2m} B_{2m}\sum_{n\leq x}\frac{\mu(n)}{n} \Delta_{-2m}\l(\frac{x}{n}\r)             \nonumber
\end{align*}
Under the hypotheses of the theorem, we use Lemma~\ref{lem33} to obtain
\begin{align*}                                                               
& M_r(x, {\rm id})  = \frac{[x]}{2}+  \frac{1}{(r+1)\zeta(2)}x\log x 
+  \frac{x}{(r+1)\zeta(2)}\l(2\gamma - 1- \frac{\zeta'(2)}{\zeta(2)} + C_{\rm{odd}}(r) \r)            \\
& + \frac{1}{r+1}\sum_{n\leq x}\frac{\mu(n)}{n}\Delta\l(\frac{x}{n}\r) 
   + \frac{1}{r+1}\sum_{m=1}^{[r/2]} {r+1 \choose 2m} B_{2m}\zeta(2m)\sum_{n\leq x}\frac{\mu(n)}{n}\Delta_{-2m}\l(\frac{x}{n}\r)\\
&  + \frac{1}{r+1}\l(2\gamma-1 + C_{\rm{odd}}(r) \r)
     \sum_{|\gamma|\leq T_{*}}\frac{x^{\rho-1}}{(\rho-2)\zeta'(\rho)}      \\
&  + \frac{1}{r+1}\sum_{|\gamma|\leq T_{*}}\frac{x^{\rho-1}}{(\rho-2)^2 \zeta'(\rho)}
   - \frac{C_{\rm{even}}(r)}{2(r+1)} \sum_{|\gamma|\leq T_{*}}\frac{x^{\rho-1}}{(\rho-1)\zeta'(\rho)}  + O_{r}\l(x^{-3}\r),     
\end{align*}
which completes the proof.
\end{proof}


\begin{proof}[Proof of Theorem \ref{th33}]
To prove our theorem it suffices to show that
$$ 
\sum_{|\gamma| \leq  T_{*}}
\frac{x^{-\frac{1}{2}+i\gamma}}{(-j+i\gamma) \zeta'(\frac{1}{2}+i\gamma)} = O\l(x^{-1/2}(\log x)^{5/4}\r)
$$ 
with $j=1/2$ and $3/2$.
We take $\lambda=-1/2$ into (\ref{GH}), then
$
\d J_{-1/2}(T_{*})  \ll  T_{*}(\log T_{*})^{1/4}.
$
Using  the above   and partial summation  we have
\begin{align*}
&\sum_{|\gamma| \leq  T_{*}}
\frac{1}{\gamma^{} |\zeta'(\frac12+i\gamma)|}
\ll \left[\frac{J_{-1/2}(t)}{t^{}}\right]_{14}^{T_{*}}
+ \int_{14}^{T_{*}}\frac{J_{-1/2}(t)}{t^{2}}dt
\ll (\log T_{*})^{5/4},
\end{align*}
and the proof is complete. 
\end{proof}


\section*{Acknowledgement}
The first author is supported by the Austrian Science
Fund (FWF), Project F5507-N26, which is part
of the Special Research Program  ``Quasi Monte
Carlo Methods: Theory and Applications''. The third author is supported by the Austrian Science Fund (FWF): Projects F5507-N26 and F5505-N26, which are part
of the Special Research Program  ``Quasi Monte
Carlo Methods: Theory and Applications''. 
 

\medskip\noindent {\footnotesize Lisa Kaltenb\"ock: 
Institute of Financial Mathematics and Applied Number Theory, Johannes Kepler University, Altenbergerstrasse 69, 4040 Linz, Austria.\\
e-mail: {\tt lisa.kaltenb\"ock@jku.at}}

\medskip\noindent {\footnotesize Isao Kiuchi: Department of Mathematical Sciences, Faculty of Science,
Yamaguchi University, Yoshida 1677-1, Yamaguchi 753-8512, Japan. \\
e-mail: {\tt kiuchi@yamaguchi-u.ac.jp}}

\medskip\noindent {\footnotesize Sumaia Saad Eddin: 
Institute of Financial Mathematics and Applied Number Theory, Johannes Kepler University, Altenbergerstrasse 69, 4040 Linz, Austria.\\
e-mail: {\tt sumaia.saad\_eddin@jku.at}}

\medskip\noindent {\footnotesize Masaaki Ueda: Department of Mathematical Sciences, Faculty of Science,
Yamaguchi University, Yoshida 1677-1, Yamaguchi 753-8512, Japan. \\
e-mail: {\tt i001wb@yamaguchi-u.ac.jp}}


\begin{thebibliography}{00}


\bibitem{Ap} \ T. M. Apostol, \ {\it Introduction to  Analytic Number Theory}, \   Springer, 1976.


\bibitem{BR}\ M. Balazard and A. de Roton,    \  Notes de lecture de l'article `` Partial sums of the 
M\"{o}bius function"  de Kannan Soundararajan, \ arXiv: 0810.3587v1 [math.NT] 20 Oct 2008. 


\bibitem{Br}\ K. A. Broughan,   \  The average order of the Dirichlet series of the gcd-sum  function,    \
{\it J. Integer  Sequences}   {\bf 10}  (2007), Article  07.4.2.


\bibitem{B}\ O. Bordell\'{e}s,  \  The composition of the gcd and certain arithmetic functions,    \
{\it J. Integer  Sequences}   {\bf 13}  (2010), Article  10.7.1.


\bibitem{G} S. M. Gonek,  On negative moments of the Riemann zeta-function, {\it Mathematika}. {\bf 36} (1989), 71--88


\bibitem{H} D. Hejhal,  On the distribution of $\log{\left|\zeta'(\frac{1}{2}+it)\right|}$, {\it Number theory,
trace formula and discrete groups, Symposium in Honor of Atle Selberg, Oslo, Norway, July 14-21, 1987}, ( edited by K.E. Aubert, E. Bombieri, and D. Goldfeld), Academic Press, San Diego, 1989, 343--370. 


\bibitem{IK}\ S. Inoue and I. Kiuchi,   
On sums of  gcd-sum functions, 
preprint.


\bibitem{IK1}\ S. Inoue and I. Kiuchi, 
On sums of  gcd-sum functions II,  
preprint.


\bibitem{Jia}\ R. Q. Jia,  Estimation of partial sums of series $\sum\mu(n)/n$,   {\it Kexue  Tongbao}.  {\bf 30}  (1985),  575--578.


\bibitem{K} I. Kiuchi,  On sums of averages of generalized Ramanujan sums, 
 {\it Tokyo  J.  Math.} {\bf 40} (2017),  255--275.


\bibitem{K1} I. Kiuchi,    Sums of averages of gcd-sum  functions, 
  {\it J.  Number Theory}  {\bf 176} (2017),  449--472.


\bibitem{K2}\ I. Kiuchi,  Sums of averages of generalized Ramanujan sums, 
  {\it J.  Number Theory} {\bf 180} (2017),  310--348.


\bibitem{KS} I. Kiuchi and S.  Saad  eddin,  Sums of weighted averages of gcd-sum functions, 
 {\it International J. Number Theory} {\bf 14} (2018), 2699--2728. 


\bibitem{L} H.-Q. Liu,   On Euler's function,  
           {\it Proc. Royal Soc. Edinburgh}  {\bf 146} (2016),  769--775.


\bibitem{MM} H. Maier and H. L. Montgomery,  On the sum of the M\"{o}bius function,   
           {\it Bull. London. Math. Soc} {\bf 41}  (2009), 213--226.


\bibitem{SS} R. Sitaramachandra Rao and D. Suryanarayana,  The number of pairs of integers with {\rm{L. C. M.}} $\leqq x$,  
          {\it Arch. Math. (Basel)} {\bf 21}  (1970), 490--497.


\bibitem{So} K. Soundararajan,  Partial sums of the M\"{o}bius function,  
          {\it J. reine angew. Math.}   {\bf 631}  (2009), 141--152.


\bibitem{TZ} Y. Tanigawa and W. Zhai,   On the gcd-sum functions,  
 {\it J. Integer  Sequences} {\bf 11}  (2008), Article  08.2.3.


\bibitem{To4} L. T\'{o}th,  A survey of  gcd--sum functions,  
 {\it J. Integer  Sequences} {\bf 13}  (2010), Article  10.8.1.


\bibitem{W}  A. Walfisz,  {\it Weylsche Exponentialsummen in der Neueren Zahlentheorie}, 
Veb  Deutscher Verlag Der Wissenschaften,  Berlin 1963.


\end{thebibliography}
\end{document}